\documentclass[a4paper, british]{amsart}

\usepackage{libertine}
\usepackage[libertine]{newtxmath}

\usepackage{babel}
\usepackage{enumitem}
\usepackage{hyperref}
\usepackage[utf8]{inputenc}
\usepackage{newunicodechar}
\usepackage{mathtools}
\usepackage{varioref}
\usepackage[arrow,curve,matrix]{xy}

\usepackage{colortbl}
\usepackage{graphicx}
\usepackage{tikz}

\usepackage{mathrsfs}

\definecolor{linkred}{rgb}{0.7,0.2,0.2}
\definecolor{linkblue}{rgb}{0,0.2,0.6}

\numberwithin{figure}{section}

\usepackage[hyperpageref]{backref}

\setdescription{labelindent=\parindent, leftmargin=2\parindent}
\setitemize[1]{labelindent=\parindent, leftmargin=2\parindent}
\setenumerate[1]{labelindent=0cm, leftmargin=*, widest=iiii}

\newunicodechar{א}{\ensuremath{\aleph}}
\newunicodechar{α}{\ensuremath{\alpha}}
\newunicodechar{β}{\ensuremath{\beta}}
\newunicodechar{χ}{\ensuremath{\chi}}
\newunicodechar{δ}{\ensuremath{\delta}}
\newunicodechar{ε}{\ensuremath{\varepsilon}}
\newunicodechar{Δ}{\ensuremath{\Delta}}
\newunicodechar{η}{\ensuremath{\eta}}
\newunicodechar{γ}{\ensuremath{\gamma}}
\newunicodechar{Γ}{\ensuremath{\Gamma}}
\newunicodechar{ι}{\ensuremath{\iota}}
\newunicodechar{κ}{\ensuremath{\kappa}}
\newunicodechar{λ}{\ensuremath{\lambda}}
\newunicodechar{Λ}{\ensuremath{\Lambda}}
\newunicodechar{ν}{\ensuremath{\nu}}
\newunicodechar{μ}{\ensuremath{\mu}}
\newunicodechar{ω}{\ensuremath{\omega}}
\newunicodechar{Ω}{\ensuremath{\Omega}}
\newunicodechar{π}{\ensuremath{\pi}}
\newunicodechar{Π}{\ensuremath{\Pi}}
\newunicodechar{φ}{\ensuremath{\phi}}
\newunicodechar{Φ}{\ensuremath{\Phi}}
\newunicodechar{ψ}{\ensuremath{\psi}}
\newunicodechar{Ψ}{\ensuremath{\Psi}}
\newunicodechar{ρ}{\ensuremath{\rho}}
\newunicodechar{σ}{\ensuremath{\sigma}}
\newunicodechar{Σ}{\ensuremath{\Sigma}}
\newunicodechar{τ}{\ensuremath{\tau}}
\newunicodechar{θ}{\ensuremath{\theta}}
\newunicodechar{Θ}{\ensuremath{\Theta}}
\newunicodechar{ξ}{\ensuremath{\xi}}
\newunicodechar{Ξ}{\ensuremath{\Xi}}
\newunicodechar{ζ}{\ensuremath{\zeta}}

\newunicodechar{ℓ}{\ensuremath{\ell}}
\newunicodechar{ï}{\"{\i}}

\newunicodechar{𝔸}{\ensuremath{\bA}}
\newunicodechar{𝔹}{\ensuremath{\bB}}
\newunicodechar{ℂ}{\ensuremath{\bC}}
\newunicodechar{𝔻}{\ensuremath{\bD}}
\newunicodechar{𝔼}{\ensuremath{\bE}}
\newunicodechar{𝔽}{\ensuremath{\bF}}
\newunicodechar{𝔾}{\ensuremath{\bG}}
\newunicodechar{ℕ}{\ensuremath{\bN}}
\newunicodechar{ℙ}{\ensuremath{\bP}}
\newunicodechar{ℚ}{\ensuremath{\bQ}}
\newunicodechar{ℝ}{\ensuremath{\bR}}
\newunicodechar{𝕏}{\ensuremath{\bX}}
\newunicodechar{ℤ}{\ensuremath{\bZ}}
\newunicodechar{𝒜}{\ensuremath{\sA}}
\newunicodechar{ℬ}{\ensuremath{\sB}}
\newunicodechar{𝒞}{\ensuremath{\sC}}
\newunicodechar{𝒟}{\ensuremath{\sD}}
\newunicodechar{ℰ}{\ensuremath{\sE}}
\newunicodechar{ℱ}{\ensuremath{\sF}}
\newunicodechar{𝒢}{\ensuremath{\sG}}
\newunicodechar{ℋ}{\ensuremath{\sH}}
\newunicodechar{𝒥}{\ensuremath{\sJ}}
\newunicodechar{ℒ}{\ensuremath{\sL}}
\newunicodechar{ℳ}{\ensuremath{\sM}}
\newunicodechar{𝒪}{\ensuremath{\sO}}
\newunicodechar{𝒬}{\ensuremath{\sQ}}
\newunicodechar{𝒮}{\ensuremath{\sS}}
\newunicodechar{𝒯}{\ensuremath{\sT}}
\newunicodechar{𝒲}{\ensuremath{\sW}}

\newunicodechar{∂}{\ensuremath{\partial}}
\newunicodechar{∇}{\ensuremath{\nabla}}

\newunicodechar{↺}{\ensuremath{\circlearrowleft}}
\newunicodechar{∞}{\ensuremath{\infty}}
\newunicodechar{⊕}{\ensuremath{\oplus}}
\newunicodechar{⊗}{\ensuremath{\otimes}}
\newunicodechar{•}{\ensuremath{\bullet}}
\newunicodechar{Λ}{\ensuremath{\wedge}}
\newunicodechar{↪}{\ensuremath{\into}}
\newunicodechar{→}{\ensuremath{\to}}
\newunicodechar{↦}{\ensuremath{\mapsto}}
\newunicodechar{⨯}{\ensuremath{\times}}
\newunicodechar{∪}{\ensuremath{\cup}}
\newunicodechar{∩}{\ensuremath{\cap}}
\newunicodechar{⊋}{\ensuremath{\supsetneq}}
\newunicodechar{⊇}{\ensuremath{\supseteq}}
\newunicodechar{⊃}{\ensuremath{\supset}}
\newunicodechar{⊊}{\ensuremath{\subsetneq}}
\newunicodechar{⊆}{\ensuremath{\subseteq}}
\newunicodechar{⊂}{\ensuremath{\subset}}
\newunicodechar{⊄}{\ensuremath{\not \subset}}
\newunicodechar{≥}{\ensuremath{\geq}}
\newunicodechar{≠}{\ensuremath{\neq}}
\newunicodechar{≫}{\ensuremath{\gg}}
\newunicodechar{≪}{\ensuremath{\ll}}

\newunicodechar{≤}{\ensuremath{\leq}}
\newunicodechar{∈}{\ensuremath{\in}}
\newunicodechar{∉}{\ensuremath{\not \in}}
\newunicodechar{∖}{\ensuremath{\setminus}}
\newunicodechar{◦}{\ensuremath{\circ}}
\newunicodechar{°}{\ensuremath{^\circ}}
\newunicodechar{…}{\ifmmode\mathellipsis\else\textellipsis\fi}
\newunicodechar{·}{\ensuremath{\cdot}}
\newunicodechar{⋯}{\ensuremath{\cdots}}
\newunicodechar{∅}{\ensuremath{\emptyset}}
\newunicodechar{⇒}{\ensuremath{\Rightarrow}}

\newunicodechar{⁰}{\ensuremath{^0}}
\newunicodechar{¹}{\ensuremath{^1}}
\newunicodechar{²}{\ensuremath{^2}}
\newunicodechar{³}{\ensuremath{^3}}
\newunicodechar{⁴}{\ensuremath{^4}}
\newunicodechar{⁵}{\ensuremath{^5}}
\newunicodechar{⁶}{\ensuremath{^6}}
\newunicodechar{⁷}{\ensuremath{^7}}
\newunicodechar{⁸}{\ensuremath{^8}}
\newunicodechar{⁹}{\ensuremath{^9}}
\newunicodechar{ⁱ}{\ensuremath{^i}}
\newunicodechar{⁺}{\ensuremath{^+}}

\newunicodechar{⌈}{\ensuremath{\lceil}}
\newunicodechar{⌉}{\ensuremath{\rceil}}
\newunicodechar{⌊}{\ensuremath{\lfloor}}
\newunicodechar{⌋}{\ensuremath{\rfloor}}

\newunicodechar{≅}{\ensuremath{\cong}}
\newunicodechar{⇔}{\ensuremath{\Leftrightarrow}}
\newunicodechar{∃}{\ensuremath{\exists}}
\newunicodechar{±}{\ensuremath{\pm}}

\DeclareFontFamily{OMS}{rsfs}{\skewchar\font'60}
\DeclareFontShape{OMS}{rsfs}{m}{n}{<-5>rsfs5 <5-7>rsfs7 <7->rsfs10 }{}
\DeclareSymbolFont{rsfs}{OMS}{rsfs}{m}{n}
\DeclareSymbolFontAlphabet{\scr}{rsfs}
\DeclareSymbolFontAlphabet{\scr}{rsfs}

\DeclareFontFamily{U}{mathx}{\hyphenchar\font45}
\DeclareFontShape{U}{mathx}{m}{n}{
      <5> <6> <7> <8> <9> <10>
      <10.95> <12> <14.4> <17.28> <20.74> <24.88>
      mathx10
      }{}
\DeclareSymbolFont{mathx}{U}{mathx}{m}{n}
\DeclareFontSubstitution{U}{mathx}{m}{n}
\DeclareMathAccent{\wcheck}{0}{mathx}{"71}

\DeclareMathOperator{\Aut}{Aut}

\DeclareMathOperator{\Pic}{Pic}
\DeclareMathOperator{\rank}{rank}

\DeclareMathOperator{\red}{red}
\DeclareMathOperator{\reg}{reg}

\DeclareMathOperator{\sing}{sing}

\newcommand{\sA}{\scr{A}}
\newcommand{\sB}{\scr{B}}
\newcommand{\sC}{\scr{C}}
\newcommand{\sD}{\scr{D}}
\newcommand{\sE}{\scr{E}}
\newcommand{\sF}{\scr{F}}
\newcommand{\sG}{\scr{G}}
\newcommand{\sH}{\scr{H}}

\newcommand{\sJ}{\scr{J}}

\newcommand{\sL}{\scr{L}}
\newcommand{\sM}{\scr{M}}

\newcommand{\sO}{\scr{O}}

\newcommand{\sQ}{\scr{Q}}

\newcommand{\sS}{\scr{S}}
\newcommand{\sT}{\scr{T}}

\newcommand{\sW}{\scr{W}}

\newcommand{\bA}{\mathbb{A}}
\newcommand{\bB}{\mathbb{B}}
\newcommand{\bC}{\mathbb{C}}
\newcommand{\bD}{\mathbb{D}}
\newcommand{\bE}{\mathbb{E}}
\newcommand{\bF}{\mathbb{F}}
\newcommand{\bG}{\mathbb{G}}

\newcommand{\bN}{\mathbb{N}}

\newcommand{\bP}{\mathbb{P}}
\newcommand{\bQ}{\mathbb{Q}}
\newcommand{\bR}{\mathbb{R}}

\newcommand{\bX}{\mathbb{X}}

\newcommand{\bZ}{\mathbb{Z}}

\theoremstyle{plain}
\newtheorem{thm}{Theorem}[section]

\newtheorem{cor}[thm]{Corollary}
\newtheorem{defn}[thm]{Definition}

\newtheorem{lem}[thm]{Lemma}

\newtheorem{prop}[thm]{Proposition}

\theoremstyle{remark}

\newtheorem{claim}[thm]{Claim}
\newtheorem{c-n-d}[thm]{Claim and Definition}

\newtheorem{notation}[thm]{Notation}

\newtheorem{rem}[thm]{Remark}
\newtheorem{question}[thm]{Question}
\newtheorem*{rem-nonumber}{Remark}

\numberwithin{equation}{thm}

\setlist[enumerate]{label=(\thethm.\arabic*), before={\setcounter{enumi}{\value{equation}}}, after={\setcounter{equation}{\value{enumi}}}}

\newcommand{\into}{\hookrightarrow}

\newcommand{\wtilde}{\widetilde}
\newcommand{\what}{\widehat}

\newcommand\CounterStep{\addtocounter{thm}{1}\setcounter{equation}{0}}

\newcommand{\factor}[2]{\left. \raise 2pt\hbox{$#1$} \right/\hskip -2pt\raise -2pt\hbox{$#2$}}
\newcommand{\Publication}[1]{}

\newcommand{\subversionInfo}{}
\newcommand{\svnid}[1]{}
\newcommand{\approvals}[2][Approval]{}
\renewcommand{\phi}{\varphi}
\allowdisplaybreaks

\usepackage{tikz-cd}
\tikzset{commutative diagrams/arrow style=tikz}

\DeclareMathOperator{\Alb}{Alb}
\DeclareMathOperator{\alb}{alb}

\DeclareMathOperator{\Div}{Div}
\DeclareMathOperator{\topsing}{sing,top}

\author{Daniel Greb}
\address{Daniel Greb, Essener Seminar für Algebraische Geometrie und Arithmetik, Fakultät für Ma\-the\-matik, Universität Duisburg--Essen, 45117 Essen, Germany}
\email{\href{mailto:daniel.greb@uni-due.de}{daniel.greb@uni-due.de}}
\urladdr{\href{https://www.esaga.uni-due.de/daniel.greb}{https://www.esaga.uni-due.de/daniel.greb}}

\author{Stefan Kebekus}
\address{Stefan Kebekus, Mathematisches Institut, Albert-Ludwigs-Universität Freiburg, Ernst-Zermelo-Straße 1, 79104 Freiburg im Breisgau, Germany}
\email{\href{mailto:stefan.kebekus@math.uni-freiburg.de}{stefan.kebekus@math.uni-freiburg.de}}
\urladdr{\href{https://cplx.vm.uni-freiburg.de}{https://cplx.vm.uni-freiburg.de}}

\author{Thomas Peternell}
\address{Thomas Peternell, Mathematisches Institut, Universität
  Bayreuth, 95440~Bayreuth, Germany}
\email{\href{mailto:thomas.peternell@uni-bayreuth.de}{thomas.peternell@uni-bayreuth.de}}
\urladdr{\href{http://www.komplexe-analysis.uni-bayreuth.de}{http://www.komplexe-analysis.uni-bayreuth.de}}

\thanks{Daniel Greb gratefully acknowledges support by the DFG-Research Training Group 2553 and the ANR-DFG project ``Quantization, Singularities and Holomorphic Dynamics''.  Stefan Kebekus gratefully acknowledges the support through a fellowship
  of the Freiburg Institute of Advanced Studies (FRIAS)}



\title[Miyaoka-Yau inequalities and topological characterization]{Miyaoka-Yau inequalities and the topological characterization of certain klt varieties}
\date{\today}

\makeatletter
\hypersetup{
  pdfauthor={\authors},
  pdftitle={\@title},
  pdfsubject={\@subjclass},
  pdfkeywords={\@keywords},
  pdfstartview={Fit},
  pdfpagelayout={TwoColumnRight},
  pdfpagemode={UseOutlines},
  bookmarks,
  colorlinks,
  linkcolor=linkblue,
  citecolor=linkred,
  urlcolor=linkred
}
\makeatother

\begin{document}

\begin{abstract}
\selectlanguage{british}

Ball quotients, hyperelliptic varieties, and projective spaces
are characterized by their Chern classes, as the varieties where the Miyaoka-Yau
inequality becomes an equality.  Ball quotients, Abelian varieties, and projective
spaces are also characterized topologically: if a complex, projective manifold
$X$ is homeomorphic to a variety of this type, then $X$ is itself of this type.
In this paper, similar results are established for projective varieties with klt
singularities that are homeomorphic to singular ball quotients, quotients of
Abelian varieties, or projective spaces.

% !TEX root = naht3
\end{abstract}
\approvals{Daniel & yes \\ Stefan & yes \\ Thomas & yes}

\maketitle
\tableofcontents

%
% Do not edit the following line.  The text is automatically updated by
% subversion.
%
\svnid{$Id: 01-intro.tex 293 2023-10-06 11:11:22Z kebekus $}
\selectlanguage{british}

\section{Introduction}
\subversionInfo
\label{sec:1}

\subsection{The Miyaoka-Yau inequality for projective manifolds}
\approvals{Daniel & yes \\ Stefan & yes \\ Thomas & yes}

Let $X$ be an $n$-di\-men\-sional complex-projective manifold and let $D$ be any
divisor on $X$.  Recall that $X$ is said to ``satisfy the Miyaoka-Yau inequality
for $D$'' if the following Chern class inequality holds,
\[
  \big( 2(n+1)·c_2(X) - n·c_1(X)² \big)·[D]^{n-2} ≥ 0.
\]
It is a classic fact that $n$-dimensional projective manifolds $X$ whose
canonical bundles are ample or trivial satisfy Miyaoka-Yau inequalities.  In
case of equality, the universal covers are of particularly simple form.

\begin{thm}[Ball quotients and hyperelliptic varieties]\label{thm:1-1}%
  Let $X$ be an $n$-dimensional complex projective manifold.
  \begin{itemize}
    \item If $K_X$ is ample, then $X$ satisfies the Miyaoka-Yau inequality for
    $K_X$.  In case of equality, the universal cover of $X$ is the unit the ball
    $𝔹^n$.

    \item If $K_X$ is trivial and $D$ is any ample divisor, then $X$ satisfies
    the Miyaoka-Yau inequality for $D$.  In case of equality, the universal
    cover of $X$ is the affine space $ℂ^n$.  \qed
  \end{itemize}
\end{thm}

We refer the reader to \cite{GKT16} for a full discussion and references to the
original literature.

In the Fano case, where $-K_X$ is ample, the situation is more complicated, due
to the fact that the tangent bundle $𝒯_X$ and the canonical extension $ℰ_X$
need not be semistable\footnote{Recall that the canonical extension $ℰ_X$ is
defined as the middle term of the exact sequence $0 → 𝒪_X → ℰ_X → 𝒯_X → 0$
whose extension class equals $c_1(X) ∈ H¹\bigl(X,\, Ω¹_X\bigr)$.}.  If $ℰ_X$ is
semistable, then analogous results hold, see \cite[Theorem 1.3]{GKP22}, as well
as further references given there.

\begin{thm}[Projective space]\label{thm:1-2}%
  Let $X$ be an $n$-dimensional projective manifold.  If $-K_X$ is ample and if
  the canonical extension is semistable with respect to $-K_X$, then $X$
  satisfies the Miyaoka-Yau inequality for $-K_X$.  In case of equality, $X$ is
  isomorphic to the projective space $ℙ^n$.  \qed
\end{thm}

In each of the three settings, the equality cases are characterized
topologically: if $M$ is any projective manifold homeomorphic to a ball
quotient, a finite étale quotient of an Abelian variety or the projective space,
then $M$ itself is biholomorphic to a ball quotient, to a finite étale quotient
of an Abelian variety, or to the projective space.  For ball quotients, this is
a theorem of Siu \cite{MR584075}.  The torus case is due to Catanese
\cite{MR1945361}, whereas the Fano case is due to Hirzebruch-Kodaira
\cite{MR92195} and Yau \cite{MR480350}.

\subsection{Spaces with MMP singularities}
\approvals{Daniel & yes \\ Stefan & yes \\ Thomas & yes}

In general, it is rarely the case that the canonical bundle of a projective
variety has a definite "sign".  Minimal model theory offers a solution to this
problem, at the expense of introducing singularities.  It is therefore natural
to extend our study from projective manifolds to projective varieties with
Kawamata log terminal (= klt) singularities.  For klt varieties whose canonical
sheaves are ample, trivial or negative, analogues of Theorems~\ref{thm:1-1} and
\ref{thm:1-2} have been found in the last few years.  We refer the reader to
\cite[Thm.~1.5]{GKPT19} for a characterization of singular ball quotients among
projective varieties with klt singularities - see Definition \ref{def:2-1} for
the notion of singular ball quotients.  Characterizations of torus quotients and
quotients of the projective space can be found in \cite{LT18},
\cite[Thm.~1.2]{MR4263792} and \cite[Thm.~1.3]{GKP22}.  In each case, we find it
striking that the Chern class equalities imply that the underlying space has no
worse than quotient singularities.

\subsection{Main results of this paper}
\approvals{Daniel & yes \\ Stefan & yes \\ Thomas & yes}

This paper asks whether the topological characterizations of ball quotients,
Abelian varieties and the projective spaces have analogues in the klt settings.
Section~\ref{sec:2} establishes a topological characterization of singular ball
quotients.  The main result of this section, Theorem~\ref{thm:2-3}, can be seen
as a direct analogue of Siu's rigidity theorems.

\begin{thm}[Rigidity in the klt setting, see Theorem~\ref{thm:2-3}]\label{thm:1-3}%
  Let $X$ be a singular quotient of an irreducible bounded symmetric domain and
  let $M$ be a normal projective variety that is homeomorphic to $X$.  If $\dim
  X ≥ 2$, then, $M$ is biholomorphic or conjugate-biholomorphic to $X$.
\end{thm}

Using somewhat different methods, Section~\ref{sec:3} generalizes Catanese's
result to the klt setting.

\begin{thm}[Varieties homeomorphic to torus quotients, see Theorem~\ref{thm:3-3}]%
  Let $M$ be a compact complex space with klt singularities.  Assume that $M$ is
  bimeromorphic to a Kähler manifold.  If $M$ is homeomorphic to a singular
  torus quotient, then $M$ is a singular torus quotient.
\end{thm}

In both cases, we find that certain Chern classes equalities are
invariant under homeomorphisms.
  
Varieties homeomorphic to projective spaces are harder to investigate.
Section~\ref{sec:4} gives a full topological characterization of $ℙ³$, but
cannot fully solve the characterization problem in higher dimensions.

\begin{thm}[Topological $ℙ³$, see Theorem~\ref{thm:4-20}]%
  Let $X$ be a projective klt variety that is homeomorphic to $ℙ³$.  Then, $X ≅
  ℙ³$.
\end{thm}

However, we present some partial results that severely restrict the geometry of
potential exotic varieties homeomorphic to $ℙ^n$.  These allow us to show the
following.

\begin{thm}[$ℚ$-Fanos in dimension 4 and 5, see Theorem~\ref{thm:4-21}]%
  Let $X$ be a projective klt variety that is homeomorphic to $ℙ^n$ with $n = 4$
  or $n = 5$.  Then, $X ≅ ℙ^n$, unless $K_X$ is ample.
\end{thm}

\subsection*{Dedication}
\approvals{Daniel & yes \\ Stefan & yes \\ Thomas & yes}

We dedicate this paper to the memory of Jean-Pierre Demailly.  His passing is a
tremendous loss to the mathematical community and to all who knew him.

\subsection*{Acknowledgements}
\approvals{Daniel & yes \\ Stefan & yes \\ Thomas & yes}

We thank Markus Banagl, Sebastian Goette, Wolfgang Lück, Jörg Schürmann and
Michael Weiss for providing detailed guidance regarding Pontrjagin classes of
topological manifolds.  Igor Belegradek kindly answered our questions on
MathOverflow.  We also thank the referee, who suggested, among other
improvements, to generalize the results of Section~\ref{sec:3} to the Kähler
case.

% !TEX root = naht3
%
% Do not edit the following line.  The text is automatically updated by
% subversion.
%
\svnid{$Id: 02-rigidity.tex 291 2023-09-25 11:44:39Z kebekus $}
\selectlanguage{british}

\section{Mostow Rigidity for singular quotients of symmetric domains}
\subversionInfo
\approvals{Daniel & yes \\ Stefan & yes \\ Thomas & yes}
\label{sec:2}

Consider a compact Kähler manifold $X$ whose universal cover is a bounded
symmetric domain.  Siu has shown in \cite[Thm.~4]{MR584075} and \cite[Main
Theorem]{Siu81} that any compact Kähler manifold $M$ which is homotopy
equivalent to $X$ is biholomorphic or conjugate-biholomorphic\footnote{See also
\cite[Sect.~7]{MR3404712} and \cite[Chapt.~5 and 6]{MR1379330} as general
references for the main ideas behind Siu's results and for related topics.} to
$X$.  We show an analogous result for homeomorphisms between \emph{singular}
varieties $M$ and $X$.  The following notion will be used.

\begin{defn}[Quasi-étale cover]%
  A finite, surjective morphism between normal, irreducible complex spaces is
  called \emph{quasi-étale cover} if it is unbranched in codimension one.
\end{defn}

\begin{defn}[Singular quotient of bounded symmetric domtain]\label{def:2-1}%
  Let $Ω$ be an irreducible bounded symmetric domain.  A normal projective
  variety $X$ is called a \emph{singular quotient of $Ω$} if there exists a
  quasi-étale cover $\what{X} → X$, where $\what{X}$ is a smooth variety whose
  universal cover is $Ω$.
\end{defn}

\begin{rem}[Singular quotients are quotients]
  Let $X$ be a singular quotient of an irreducible bounded symmetric domain $Ω$.
  Passing to a suitable Galois closure, one finds a quasi-étale \emph{Galois}
  cover $\what{X} → X$, where $\what{X}$ is a smooth variety whose universal
  cover is $Ω$.  In particular, it follows that $X$ is a quotient variety and
  that it has quotient singularities.  Moreover, it can be shown as in
  \cite[Sect.~9]{GKPT19b} that $X$ is actually a quotient of $Ω$ by the
  fundamental group of $X_{\reg}$, which acts properly discontinously on $Ω$.
  In addition, the action is free in codimension one.
\end{rem}

\begin{thm}[Mostow rigidity in the klt setting]\label{thm:2-3}%
  Let $X$ be a singular quotient of an irreducible bounded symmetric domain and
  let $M$ be a normal projective variety that is homeomorphic to $X$.  If $\dim
  X ≥ 2$, then, $M$ is biholomorphic or conjugate-biholomorphic to $X$.
\end{thm}

\begin{rem}[Varieties conjugate-biholomorphic to ball quotients]
  We are particularly interested in the case where the bounded symmetric domain
  of Theorem~\ref{thm:2-3} is the unit ball.  For this, observe that the set of
  (singular) ball quotients is invariant under conjugation.  It follows that if
  the variety $M$ of Theorem~\ref{thm:2-3} is biholomorphic or
  conjugate-biholomorphic to a (singular) ball quotient $X$, then $M$ is itself
  a (singular) ball quotient.
\end{rem}

Before proving Theorem~\ref{thm:2-3} in Sections~\ref{sec:2-1}--\ref{sec:2-3}
below, we note a first application: the Miyaoka-Yau Equality is a topological
property.  The symbols $\what{c}_•(X)$ in Corollary~\ref{cor:2-5} are the
$ℚ$-Chern classes of the klt space $X$, as defined and discussed for instance
\cite[Sect.~3.7]{GKPT19b}.

\begin{cor}[Topological invariance of the Miyaoka-Yau equality]\label{cor:2-5}%
  Let $X$ be a projective klt variety with $K_X$ ample.  Assume that the
  Miyaoka-Yau equality holds:
  \[
    \bigl( 2(n+1)· \what{c}_2(𝒯_X) - n · \what{c}_1(𝒯_X)² \bigr) · [K_X]^{n-2} =0.
  \]
  Let $M$ be a normal projective variety homeomorphic to $X$.  Then $M$ is klt,
  $K_M$ is ample and
  \[
    \bigl( 2(n+1)· \what{c}_2(𝒯_M) - n · \what{c}_1(𝒯_M)² \bigr) · [K_M]^{n-2} =0.
  \]
\end{cor}
\begin{proof}
  Since the Miyaoka-Yau Equality holds on $X$, there is a quasi-étale cover
  $\wtilde{X} → X$ such that the universal cover of $\wtilde{X}$ is the ball,
  \cite{GKPT19b}.  By Theorem~\ref{thm:2-3}, there is a quasi-étale cover
  $\wtilde{M} → M$ such that $\wtilde{M} ≅ \wtilde{X}$ biholomorphically or
  conjugate-biholomorphically.  Hence, the universal cover of $\wtilde{M}$ is
  the ball.  It follows that $M$ is klt, $K_M$ is ample, and that the
  Miyaoka-Yau Equality holds on $M$.
\end{proof}

\subsection{Preparation for the proof of Theorem~\ref*{thm:2-3}}
\approvals{Daniel & yes \\ Stefan & yes \\ Thomas & yes}
\label{sec:2-1}

The following lemma of independent interest might be well-known.  We include a
full proof for lack of a good reference.

\begin{lem}\label{lem:2-6}%
  Let $X$ be a normal complex space.  Then, the set $X_{\topsing} ⊂ X$ of
  topological singularities is a complex-analytic set.
\end{lem}
\begin{proof}
  Recall from \cite[Thm.~on p.~43]{GoreskyMacPherson} that $X$ admits a Whitney
  stratification where all strata are locally closed complex-analytic
  submanifolds of $X$.  Recall from \cite[Chapt.~IV.8]{MR1131081} that the
  closures of the strata are complex-analytic subsets of $X$.  Since Whitney
  stratifications are locally topologically trivial along the
  strata\footnote{See \cite[Part~I, Sect.~1.4]{GoreskyMacPherson} for a detailed
  discussion.}, it follows that $X_{\topsing}$ is locally the union of finitely
  many strata.  The additional observation that the set of topologically smooth
  points, $X ∖ X_{\topsing}$, is open in the Euclidean topology implies that
  $X_{\topsing}$ is locally the union of the closures of finitely many strata,
  hence analytic.
\end{proof}

\subsection{Proof of Theorem~\ref*{thm:2-3} if $X$ is smooth}
\approvals{Daniel & yes \\ Stefan & yes \\ Thomas & yes}
\label{sec:2-2}
\CounterStep

We maintain the notation of Theorem~\ref{thm:2-3} in this section and assume
additionally that $X$ is smooth.  To begin, fix a homeomorphism $f: M → X$ and
choose a resolution of singularities, say $π: \wtilde{M} → M$.  The composed map
$g = f◦π$ is continuous and induces an isomorphism
\begin{equation}\label{eq:2-7-1}
  g_*: H_{2n} \bigl(\wtilde{M},\, ℤ\bigr) → H_{2n}\bigl(X,\, ℤ\bigr).
\end{equation}
Hence, by Siu's general rigidity result \cite[Thm.~6]{MR584075} in combination
with the curvature computations for the classical, respectively exceptional
Hermitian symmetric domains done in \cite{MR584075, Siu81}, the continuous map
$g$ is homotopic to a holomorphic or conjugate-holomorphic map $\wtilde{g}:
\wtilde{M} → X$.  Replacing the complex structure on $X$ by the conjugate
complex structure, if necessary, we may assume without loss of generality that
$\wtilde{g}$ is holomorphic and hence in particular algebraic.  The isomorphism
\eqref{eq:2-7-1} maps the fundamental class of $\wtilde{M}$ to the fundamental
class of $X$, and $\wtilde{g}$ is hence birational.

We claim that the bimeromorphic morphism $\wtilde{g}$ factors via $π$.  To
begin, observe that since $g$ contracts the fibres of $π$ and since $\wtilde{g}$
is homotopic to $g$, the map $\wtilde{g}$ contracts the fibres of $π$ as well.
In fact, given any curve $\wtilde{C} ⊂ \wtilde{M}$ with $π(\wtilde{C})$ a point,
consider its fundamental class $[\wtilde{C}] ∈ H_2\bigl(\wtilde{M},\, ℝ\bigr)$.
By assumption, we find that
\[
  \wtilde{g}_* \bigl([\wtilde{C}]\bigr) = g_* \bigl([\wtilde{C}]\bigr) = 0 ∈ H_2\bigl(X,\, ℝ\bigr).
\]
Given that $X$ is projective, this is only possible if $\wtilde{g}(\wtilde{C})$
is a point.  Since $M$ is normal and since $\wtilde{g}$ contracts the
(connected) fibres of the resolution map $π$, we obtain the desired
factorisation of $\wtilde{g}$, as follows
\[
  \begin{tikzcd}[column sep=1.5cm]
    \wtilde{M} \ar[r, "π"'] \ar[rr, bend left=15, "\wtilde{g}"] & M \ar[r, "∃!  \wtilde{f}"'] & X.
  \end{tikzcd}
\]

We claim that the birational map $\wtilde{f}$ is
biholomorphic.\footnote{Cf.~\cite[Rem.~86(2)]{MR3404712}}  By Zariski's Main
Theorem, \cite[V~Thm.~5.2]{Ha77}, it suffices to verify that it does not
contract any curve $C ⊂ M$.  Aiming for a contradiction, assume that there
exists a curve $\wtilde{C} ⊂ \wtilde{M}$ whose image $C := π(\wtilde{C})$ is a
curve in $M$, while $\wtilde{g}(\wtilde{C}) = \wtilde{f}(C) = (*)$ is a point in
$X$.  Let $d >0$ be the degree of the restricted map $π|_{\wtilde{C}}:
\wtilde{C} → C$.  Then, on the one hand,
\[
  f_*\bigl(d·[C]\bigr) = f_*\bigl(π_*[\wtilde{C}]\bigr)
  = g_* [\wtilde{C}]
  = \wtilde{g}_* [\wtilde{C}]
  = 0 ∈ H_2\bigl(X,\, ℝ\bigr).
\]
On the other hand, projectivity of $M$ implies that $d·[C]$ is a non-trivial
element of $H_2\bigl(M,\, ℝ\bigr)$, which therefore must be mapped to a
non-trivial element of $H_2\bigl(X,\, ℝ\bigr)$, since $f$ is assumed to be a
homeomorphism.  This finishes the proof of Theorem~\ref{thm:2-3} in the case
where $X$ is smooth.  \qed

\subsection{Proof of Theorem~\ref*{thm:2-3} in general}
\approvals{Daniel & yes \\ Stefan & yes \\ Thomas & yes}
\label{sec:2-3}
\CounterStep

Maintain the setting of Theorem~\ref{thm:2-3}.

\subsubsection*{Step 1: Setup}
\approvals{Daniel & yes \\ Stefan & yes \\ Thomas & yes}

By assumption, there exists a bounded symmetric domain $Ω$ and a quasi-étale
cover $τ_X :\what{X} → X$ such that the universal cover of $\what{X}$ is $Ω$.
Choose a homeomorphism $f: M → X$ and let $\what{M} := \what{X} ⨯_X M$ be the
topological fibre product.  The situation is summarized in the following
commutative diagram,
\begin{equation}\label{eq:2-8-1}
  \begin{tikzcd}[column sep=2cm]
    \what{M} \ar[r, two heads, "τ_M"] \ar[d, "\simeq"'] & M \ar[d, "f", "\simeq"'] \\
    \what{X} \ar[r, two heads, "τ_X\text{, quasi-étale}"'] & X,
  \end{tikzcd}
\end{equation}
in which the vertical maps are homeomorphisms and the horizontal maps are
surjective with finite fibres.

\subsubsection*{Step 2: A complex structure on $\what{M}$}
\approvals{Daniel & yes \\ Stefan & yes \\ Thomas & yes}

The spaces $M$, $\what{X}$ and $X$ all carry complex structures.  We aim to
equip $\what{M}$ with a structure so that all horizontal arrows in
\eqref{eq:2-8-1} become holomorphic.

\begin{claim}\label{claim:2-9}%
  There exists a normal complex structure on $\what{M}$ that makes $τ_M$ a
  finite, holomorphic, and quasi-étale cover.
\end{claim}
\begin{proof}[Proof of Claim~\ref{claim:2-9}]
  Let $X_0$ be the smooth locus of $X$, set $M_0 := f^{-1}(X_0)$ and $\what{M}_0
  := τ_M^{-1}(M_0)$.  The map $τ_M|_{\what{M}_0}$ being a local homeomorphism,
  there is a uniquely determined complex structure on $\what{M}_0$ such that
  $τ_M|_{ \what{M}_0}: \what{M}_0 → M_0$ is a finite holomorphic cover.  Since
  $X$ has quotient singularities, the topological and holomorphic singularities
  agree, $X_{\topsing} = X_{\sing}$.  Hence, $f$ being a homeomorphism, we note
  that
  \[
    M_{\topsing} = f^{-1}\left(X_{\sing}\right)%
    \quad\text{and}\quad
    M ∖ M_0 = M_{\topsing}.
  \]
  We have seen in Lemma~\ref{lem:2-6} that $M_{\topsing}$ is an analytic set.
  Therefore, by \cite[Thm.~3.4]{DethloffGrauert} and \cite[Satz~1]{MR83045}, the
  complex structure on $\what{M}_0$ uniquely extends to a normal complex
  structure on the topological manifold $\what{M}$, making $τ_M$ holomorphic and
  finite.  The branch locus of $τ_M$ has the same topological dimension as the
  branch locus of $τ_X$, so that $τ_M$ is quasi-étale, as claimed.
  \qedhere~(Claim~\ref{claim:2-9})
\end{proof}

Note that as a finite cover of the projective variety $M$, the normal complex
space $\what{M}$ is again projective.

\subsubsection*{Step 3: $\what{M}$ as a quotient of $Ω$}
\approvals{Daniel & yes \\ Stefan & yes \\ Thomas & yes}
\CounterStep

The homeomorphic varieties $\what{X}$ and $\what{M}$ reproduce the assumptions
of Theorem~\ref{thm:2-3}.  The partial results of Section~\ref{sec:2-2}
therefore apply to show that the complex spaces $\what{M}$ and $\what{X}$ are
biholomorphic or conjugate-biholomorphic.  Replacing the complex structures on
$M$ and $\what{M}$ by their conjugates, if necessary, we assume without loss of
generality for the remainder of this proof that $\what{M}$ and $\what{X}$ are
biholomorphic.  This has two consequences.
\begin{enumerate}
  \item The projective variety $\what{M}$ is smooth.  The universal cover of
  $\what{M}$ is biholomorphic to $Ω$.

  \item \label{il:2-10-2} Its quotient $M$ is a singular quotient of $Ω$ and has
  only quotient singularities.
\end{enumerate}
Recalling that quotient singularities are not topologically smooth,
Item~\ref{il:2-10-2} implies that the homeomorphism $f: M → X$ restricts to a
homeomorphism between the smooth loci, $X_{\reg}$ and $M_{\reg}$.  The situation
is summarized in the following commutative diagram,
\[
  \begin{tikzcd}[column sep=2.2cm]
    Ω \ar[r, two heads, "u_X\text{, univ.~cover}"] \ar[d, "\simeq"'] & \what{M} \ar[r, two heads, "τ_M\text{, quasi-étale}"] \ar[d, "\simeq"'] & M \ar[d, "f", "\simeq"'] & M_{\reg} \ar[l, hook, "\text{inclusion}"'] \ar[d, "f|_{X_{\reg}}", "\simeq"'] \\
    Ω \ar[r, two heads, "u_M\text{, univ.~cover}"'] & \what{X} \ar[r, two heads, "τ_X\text{, quasi-étale}"'] & X & X_{\reg}, \ar[l, hook, "\text{inclusion}"]
  \end{tikzcd}
\]
where all horizontal maps are holomorphic, and all vertical maps are
homeomorphic.

\medskip

The description of $M$ as a singular quotient of $Ω$ can be made precise.  The
argument in \cite[Sect.~9.1]{GKPT19b} shows that the fundamental group
$π_1(M_{\reg})$ acts properly discontinously on $Ω$ with quotient $M$.  In
particular, we have an injective homomorphism from $π_1(M_{\reg})$ into the
holomorphic automorphism group $\Aut(Ω)$ of $Ω$, with image a discrete cocompact
subgroup $Γ_M ⊆ \Aut(Ω)$.

The same reasoning also applies to $X$ and presents $X$ as a quotient $X =
Ω/π_1(X_{\reg})$, where $π_1(X_{\reg})$ again acts via an injective homomorphism
$π_1(X_{\reg}) ↪ \Aut(Ω)$, with image a cocompact, discrete subgroup $Γ_X$ of
$\Aut(Ω)$.

As we have seen above, $f$ induces a homeomorphism from $M_{\reg}$ to
$X_{\red}$, from which we obtain an abstract group isomorphism $θ: Γ_M → Γ_X$.

\subsubsection*{Step 4: End of proof}
\approvals{Daniel & yes \\ Stefan & yes \\ Thomas & yes}

In the remainder of the proof we will show that not only $\what{M}$ and
$\what{X}$ are (conjugate-)-biholomorphic, but that this actually holds for $M$
and $X$.  This will be a consequence of Mostow's rigidity theorem for lattices
in connected semisimple real Lie groups.  As the groups appearing in our
situation are not necessarily connected, we have to do some work to reduce to
the connected case\footnote{Alternatively, one could trace the finite group
actions through the proof of the results used in Section~\ref{sec:2-2}.}.

Given that $Ω$ is an irreducible Hermitian symmetric domain of dimension greater
than one, the identity component $\Aut°(Ω) ⊆ \Aut(Ω)$ coincides with the
identity component $I°(Ω)$ of the isometry group $I(Ω)$ of the Riemannian
symmetric space $Ω$, \cite[VIII.Lem.~4.3]{MR1834454}\footnote{As $Ω$ is
irreducible, the compatible Riemannian metric on $Ω$ is unique up to a positive
real multiple that does not change the isometry group.}, which is a non-compact
simple Lie group without non-trivial proper compact normal subgroups and with
trivial centre, \cite[Prop.~2.1.1 and bottom of p.~379]{MR1441541}.  We also
note that a Bergman-metric argument shows that $\Aut(Ω)$ is contained in $I(Ω)$.
Furthermore, both Lie groups have only finitely many connected components.

\begin{claim}\label{claim:2-11}%
  There exists an isometry $F ∈ I(Ω)$ such that
  \begin{equation}\label{eq:2-11-1}
    F◦γ = θ(γ)◦F, \quad \text{for every } γ ∈ Γ_M.
  \end{equation}
\end{claim}
\begin{proof}[Proof of Claim~\ref{claim:2-11}]
  If the rank of $Ω$ is equal to one, then $Ω ≅ 𝔹_n$, the unit ball in $𝔹^n$,
  see \cite[Sect.~X, \S6.3/4]{MR1834454}.  Consequently, the group $\Aut(Ω)$ is
  connected, and we may apply \cite[Thm.~A' on p.~4]{MR0385004} to obtain an
  automorphism of real Lie groups, $Θ : \Aut(Ω) → \Aut(Ω)$ such that $Θ|_{Γ_M} =
  θ$.  The desired isometry is then produced by an application of
  \cite[Prop.~3.9.11]{MR1441541}.

  We consider the case $\rank(Ω) ≥ 2$ for the remainder of the present proof,
  where the automorphism group may be non-connected.  To deal with this slight
  difficulty, we proceed as in \cite[p.~379f]{MR1441541}: as $\Aut(Ω)$ has
  finitely many connected components, we may assume that the subgroups
  $Γ_{\what{M}} ⊆ Γ_M$ and $Γ_{\what{X}} ⊆ Γ_X$ corresponding to the deck
  transformation groups of $u_M$ and $u_X$, respectively, are contained in the
  identity component $I°(Ω) = \Aut°(Ω)$.  Again, apply \cite[Thm.~A' on
  p.~4]{MR0385004} to obtain an automorphism of real Lie groups $Θ : I°(Ω) →
  I°(Ω)$ such that $Θ|_{Γ_{\what{M}}} = θ|_{Γ_{\what{M}}}$ and then
  \cite[Prop.~3.9.11]{MR1441541} to obtain an isometry $F ∈ I(Ω)$ such that
  \begin{equation*}
    F ◦ g = Θ(g) ◦ F, \quad \text{for every } g ∈ I°(Ω).
  \end{equation*}
  This in particular yields \eqref{eq:2-11-1} for all $γ$ contained in the
  finite index subgroup $Γ_{\what{M}}$ of $Γ_M$.  This is not yet enough.

  However, noticing that for any finite index subgroup $Γ'_M < Γ_M$, every
  $Γ'_M$-periodic vector in the sense of \cite[Def.~4.5.13]{MR1441541} by
  definition is also $Γ_M$-periodic, we see with the argument given in
  \cite[p.~379]{MR1441541}, which uses essentially the same notation as we have
  introduced here, that the set of $Γ_M$-periodic vectors is dense in the unit
  sphere bundle $SΩ$ of $Ω$.  The subsequent argument in \cite[bottom of p.~379
  and upper part of p.~380]{MR1441541} then applies verbatim to yield the
  desired relation \eqref{eq:2-11-1} for all $γ ∈ Γ_M$; this is \cite[equation
  (5) on p.~380]{MR1441541}.  \qedhere~(Claim~\ref{claim:2-11})
\end{proof}

Now, since the Hermitian symmetric domain $Ω$ is assumed to be irreducible, the
$Γ$-equivariant isometry $F ∈ I(Ω)$ is either holomorphic or
conjugate-holomorphic, as follows for example from~\cite{363432} together with
\cite[VIII.Prop.~4.2]{MR1834454}.  By the universal property of the quotient map
$π$ with respect to $Γ$-invariant holomorphic maps, $F$ hence descends to a
holomorphic or conjugate-holomorphic isomorphism from $M$ to $X$.  This
completes the proof of Theorem~\ref{thm:2-3}.  \qed

% !TEX root = naht3
%
% Do not edit the following line.  The text is automatically updated by
% subversion.
%
\svnid{$Id: 03-flat.tex 290 2023-09-25 11:40:43Z peternell $}
\selectlanguage{british}

\section{Topological characterization of torus quotients}
\subversionInfo
\approvals{Daniel & yes \\ Stefan & yes \\ Thomas & yes}
\label{sec:3}

In line with the results of Section~\ref{sec:2}, we show that a Kähler space
with klt singularities is a singular torus quotient if and only if it is
homeomorphic to a singular torus quotient.  In the smooth case, this was shown
by Catanese \cite{MR1945361}, but see also \cite[Thm.~2.2]{MR2295551}.  The
following notion is a direct analogue of Definition~\ref{def:2-1} above.

\begin{defn}[Singular torus quotient]
  A normal complex space $X$ is called a \emph{singular torus quotient} if there
  exists a quasi-étale cover $\what{X} → X$, where $\what{X}$ is a compact
  complex torus.
\end{defn}

\begin{rem}[Singular torus quotients are quotients]
  Let $X$ be a singular torus quotient.  Passing to a suitable Galois closure,
  one finds a quasi-étale \emph{Galois} cover $\what{X} → X$, where $\what{X}$
  is a compact torus.
\end{rem}

\begin{thm}[Varieties homeomorphic to torus quotients]\label{thm:3-3}%
  Let $M$ be a compact complex space with klt singularities.  Assume that $M$ is
  bimeromorphic to a Kähler manifold.  If $M$ is homeomorphic to a singular
  torus quotient, then $M$ is a singular torus quotient.
\end{thm}

Theorem~\ref{thm:3-3} will be shown in Sections~\ref{sec:3-1}--\ref{sec:3-2}
below.  In analogy to Corollary~\ref{cor:2-5} above, we note that vanishing of
$ℚ$-Chern classes is a topological property among compact Kähler spaces with klt
singularities.

\begin{cor}[Topological invariance of vanishing Chern classes]
  Let $X$ be a compact Kähler space with klt singularities.  Assume that the
  canonical class vanishes numerically, $K_X \equiv 0$, and that the second
  $ℚ$-Chern class of $𝒯_X$ satisfies
  \[
    \what{c}_2 (𝒯_X)· α_1 ⋯ α_{\dim X-2}= 0,
  \]
  for every $(\dim X-2)$-tuple of Kähler classes on $X$.  If $M$ is any compact
  Kähler space with klt singularities that is homeomorphic to $X$, then $K_M
  \equiv 0$, and the second $ℚ$-Chern class of $𝒯_M$ satisfies
  \[
    \what{c}_2 (𝒯_M)· β_1 ⋯ β_{\dim M-2} = 0,
  \]
  for every $(\dim X-2)$-tuple of Kähler classes on $M$.
\end{cor}
\begin{proof}
  The characterization of singular torus quotients in terms of Chern classes by
  Claudon, Graf and Guenancia, \cite[Cor.~1.7]{CGG23}, guarantees that $X$ is a
  torus quotient\footnote{See \cite[Thm.~1.2]{LT18} for the projective case and
  see \cite[Thm.~1.17]{GKP13} for the case where $X$ is projective and smooth in
  codimension two.}.  By Theorem~\ref{thm:3-3}, then so is $M$.
\end{proof}

\subsection{Proof of Theorem~\ref*{thm:3-3} if $M$ is homeomorphic to a torus}
\approvals{Daniel & yes \\ Stefan & yes \\ Thomas & yes}
\label{sec:3-1}

As before, we prove Theorem~\ref{thm:3-3} first in case where the (potentially
singular) space $M$ is homeomorphic to a torus.  Recalling that klt
singularities are rational, see \cite[Thm.~5.22]{KM98} for the algebraic case
and \cite[Thm.~3.12]{FujinoMMP} (together  with the vanishing theorems proven in
\cite{FujinoVanishing}) for the analytic case, we show the following, slightly
stronger statement.

\begin{prop}\label{prop:3-5}%
  Let $M$ be a compact complex space with rational singularities.  Assume that
  $M$ is bimeromorphic to a Kähler manifold.  If $M$ is homotopy equivalent to a
  compact torus, then $M$ is a compact torus.
\end{prop}
\begin{proof}
  We follow the arguments of Catanese, \cite[Thm.~4.8]{MR1945361}, and choose a
  resolution of singularities, $π : \wtilde{M} → M$, which owing to the
  assumptions on $M$ we may assume to be a compact Kähler manifold.  Using the
  assumption that $M$ has rational singularities together with the push-forward
  of the exponential sequence, we observe that the pull-back map $H¹\bigl( M,\,
  ℤ \bigr) → H¹\bigl( \wtilde M,\, ℤ \bigr)$ is an isomorphism.  In particular,
  first Betti numbers of $M$ and $\wtilde{M}$ agree.  As a next step, consider
  the Albanese map of $\wtilde{M}$, observing that $\wtilde{M}$ is bimeromorphic
  to a Kähler manifold since $M$ is.  Again using that $M$ has rational
  singularities, recall from \cite[Prop.~2.3]{Reid83a} that the Albanese factors
  via $M$,
  \[
    \begin{tikzcd}[row sep=1cm]
      \wtilde{M} \ar[r, "\alb"] \ar[d, two heads, "π\text{, resolution}"'] & \Alb.  \\
      M \ar[ur, bend right=10, "α"']
    \end{tikzcd}
  \]
  Since the pull-back morphisms
  \[
    \begin{matrix}
      \alb^* &=& π^* ◦ α^* &:& H¹\bigl( \Alb,\, ℤ \bigr) &→& H¹\bigl( \wtilde{M},\, ℤ \bigr) \\
      && π^* &:& H¹\bigl( M,\, ℤ \bigr) &→& H¹\bigl( \wtilde{M},\, ℤ \bigr)
    \end{matrix}
  \]
  are both isomorphic, we find that $α^*: H¹\bigl( \Alb,\, ℤ \bigr) → H¹\bigl(
  M,\, ℤ \bigr)$ must likewise be an isomorphism.  There is more that we can
  say.  Since the topological cohomology ring of a torus is an exterior algebra,
  \[
    H^*\bigl( \Alb,\, ℤ\bigr) = Λ^* H¹\bigl( \Alb,\, ℤ\bigr)
    \quad\text{and}\quad
    H^*\bigl(M,\, ℤ \bigr) = Λ^* H¹\bigl( M,\, ℤ\bigr),
  \]
  we find that all pull-back morphisms are isomorphisms,
  \[
    α^*: H^q\bigl( \Alb,\, ℤ
    \bigr) \xrightarrow{≅} H^q\bigl( M,\, ℤ \bigr), \quad \text{for every } 0 ≤ q ≤ 2·\dim M.
  \]
  Applying this to $q = 2·\dim M$, we see $α$ is surjective of degree one, hence
  birational.  Again, more is true: if $α$ failed to be isomorphic, Zariski's
  Main Theorem would guarantee that $α$ contracts a positive-dimensional
  subvariety, so $b_2(M) > b_2(\Alb)$.  But we have seen above that equality
  holds and hence reached a contradiction.
\end{proof}

\subsection{Proof of Theorem~\ref*{thm:3-3} in general}
\approvals{Daniel & yes \\ Stefan & yes \\ Thomas & yes}
\label{sec:3-2}

By assumption, there exists a homeomorphism $f : M → X$, where $X$ is a singular
torus quotient.  Choose a quasi-étale cover $τ_X:~\what{X} → X$, where
$\what{X}$ is a complex torus, and proceed as in the proof of
Theorem~\ref{thm:2-3}, in order to construct a diagram of continuous mappings
between normal complex spaces,
\[
  \begin{tikzcd}[column sep=2.2cm]
    \what{M} \ar[r, two heads, "τ_M\text{, quasi-étale}"] \ar[d, "≅"'] & M \ar[d, "f", "≅"'] \\
    \what{X} \ar[r, two heads, "τ_X\text{, quasi-étale}"'] & X,
  \end{tikzcd}
\]
where
\begin{itemize}
  \item the vertical maps are homeomorphisms, and

  \item the horizontal maps are holomorphic, surjective, and finite.
\end{itemize}
Since $M$ is bimeromorphic to a Kähler manifold, so is $\what{M}$.  Recalling
from \cite[Prop.~5.20]{KM98} that also $\what{M}$ has no worse than klt
singularities, Proposition~\ref{prop:3-5} will then guarantee that $\what{M}$ is
a complex torus, as claimed.  \qed

% !TEX root = naht3
%
% Do not edit the following line.  The text is automatically updated by
% subversion.
%
\svnid{$Id: 04-klt.tex 294 2023-10-11 18:42:37Z greb $}
\selectlanguage{british}

\section{Rigidity results for projective spaces}
\subversionInfo
\approvals{Daniel & yes \\ Stefan & yes \\ Thomas & yes}
\label{sec:4}

Recall the classical theorem of Hirzebruch-Kodaira, which asserts that the
projective space carries a unique structure as a Kähler manifold.

\begin{thm}[\protect{Rigidity of the projective space, \cite[p.~367]{MR92195}}]\label{thm:4-1}%
  Let $X$ be a compact Kähler manifold.  If $X$ is homeomorphic to $ℙ^n$, then
  $X$ is biholomorphic to $ℙ^n$.  \qed
\end{thm}

\begin{rem}
  Strictly speaking, Hirzebruch-Kodaira proved a somewhat weaker result: $X$ is
  biholomorphic to $ℙ^n$ if either $n$ is odd, or if $n$ is even and $c_1(X) ≠
  -(n+1)·g$, where $g$ is a generator of $H²\bigl(X,\, ℤ\bigr)$ and the
  fundamental class of a Kähler metric on $X$.  The second case was later ruled
  out by Yau's solution to the Calabi conjecture, which implies that then the
  universal cover of $X$ is the ball, contradicting $π_1(X) = 0$.

  Since the topological invariance of the Pontrjagin classes, \cite{Nov65}, was
  not known at that time, Hirzebruch-Kodaira also had to assume that $X$ is
  diffeomorphic to $ℙ^n$ rather than merely homeomorphic.
\end{rem}

We ask whether an analogue of Hirzebruch-Kodaira's theorem remains true in the
context of minimal model theory.

\begin{question}\label{q:4-3}%
  Let $X$ be a projective variety with klt singularities.  Assume that $X$ is
  homeomorphic to $ℙ^n$.  Is $X$ then biholomorphic to $ℙ^n$?
\end{question}

\subsection{Varieties homeomorphic to projective space}
\approvals{Daniel & yes \\ Stefan & yes \\ Thomas & yes}

We do not have a full answer to Question~\ref{q:4-3}.  The following proposition
will, however, restrict the geometry of potential varieties substantially.  It
will later be used to answer Question~\ref{q:4-3} in a number of special
settings.

\begin{prop}[Varieties homeomorphic to $ℙ^n$]\label{prop:4-4}%
  Let $X$ be a projective klt variety.  If $X$ is homeomorphic to $ℙ^n$, then
  the following holds.
  \begin{enumerate}
    \item\label{il:4-4-1} We have $H^q \bigl(X,\, 𝒪_X \bigr) = 0$ for every $1
    ≤ q$.

    \item\label{il:4-4-2} The Chern class map $c_1 : \Pic(X) → H²(X,ℤ) ≅ ℤ$ is
      an isomorphism.

    \item\label{il:4-4-3} The variety $X$ is smooth in codimension two.

    \item\label{il:4-4-4} The maps $r_q : H^q\bigl( X,\, ℤ \bigr) →
      H^q\bigl(X_{\reg},\, ℤ\bigr)$ are isomorphic, for every $0 ≤ q ≤ 4$.  The
      same statement holds for $ℤ_2$ coefficients.

    \item\label{il:4-4-5} Every Weil divisor on $X$ is Cartier, i.e., $X$ is
      factorial.  In particular, $X$ is Gorenstein.

    \item\label{il:4-4-6} The canonical divisor $K_X $ is ample or anti-ample.
  \end{enumerate}
\end{prop}

\begin{proof}
  We prove the items of Proposition~\ref{prop:4-4} separately.

  \subsubsection*{Item~\ref{il:4-4-1}} This is a consequence of the rationality
  of the singularities of $X$ and the isomorphisms $H^q\bigl(X,\,ℂ \bigr) \simeq
  H^q\bigl(ℙ^n,\, ℂ\bigr)$.  In fact, since $X$ has rational singularities, the
  morphisms
  \[
    \varphi_q: H^q\bigl(X,\, ℂ\bigr) → H^q\bigl(X,\, 𝒪_X\bigr)
  \]
  induced by the canonical inclusion $ℂ → 𝒪_X$, are surjective,
  \cite[Thm.~12.3]{MR1341589}.  If $q$ is odd, this already implies that
  $H^q(X,𝒪_X) = 0$.  If $q$ is even, it suffices to note that $\varphi_q$ has a
  non-trivial kernel.  For this, choose an ample line bundle $ℒ ∈ \Pic(X)$ and
  observe that
  \[
    \varphi_q\bigl(c_1(ℒ)^{q/2}\bigr) = 0 ∈ H^q\bigl(X,\, 𝒪_X\bigr).
  \]
  To prove the observation, pass to a desingularisation and use the Hodge
  decomposition there.

  \subsubsection*{Item~\ref{il:4-4-2}} The description of $c_1$ follows from
  \ref{il:4-4-1} and the exponential sequence.

  \subsubsection*{Item~\ref{il:4-4-3}} Recall that klt varieties have quotient
  singularities in codimension two, \cite[Prop.~9.3]{GKKP11}.  Smoothness
  follows because quotient singularities have non-trivial local fundamental
  groups and are hence not topologically smooth.

  \subsubsection*{Item~\ref{il:4-4-4}} We describe the relevant cohomology
  groups in terms of Borel-Moore homology, \cite{MR131271} and also refer to the
  reader to \cite[Sect.~19.1]{Fulton98} for a summary of the relevant facts
  (over $ℤ$).  The assumption that $X$ is homeomorphic to an oriented,
  connected, real manifold implies that singular cohomology and Borel-Moore
  homology agree, \cite[Thm.  7.6]{MR131271} and \cite[p.~371]{Fulton98}.  The
  same holds for the non-compact manifold $X_{\reg}$, i.e., for $R = ℤ, ℤ_2$ we
  have
  \[
    H^q\bigl( X,\, R \bigr) = H^{BM}_{2·n-q}\bigl( X,\, R \bigr) %
    \quad\text{and}\quad %
    H^q\bigl( X_{\reg},\, R \bigr) = H^{BM}_{2·n-q}\bigl( X_{\reg},\, R \bigr), %
    \quad\text{for every }q.
  \]
  The isomorphisms identify the restriction maps $r_q$ with the pull-back maps
  for Borel-Moore homology.  These feature in the localization sequence for
  Borel-Moore homology, \cite[Thm.3.8]{MR131271},
  \[
    ⋯ → H^{BM}_{2·n-q}\bigl( X_{\sing},\, R \bigr) → H^{BM}_{2·n-q}\bigl( X,\, R \bigr) \xrightarrow{r_q} H^{BM}_{2·n-q}\bigl( X_{\reg},\, R \bigr) → H^{BM}_{2·n-q-1}\bigl( X_{\sing},\, R \bigr) → ⋯
  \]
  Recalling from \cite[Lem.~19.1.1]{Fulton98} that $H^{BM}_i\bigl( X_{\sing},\,
  ℤ \bigr) = 0$ for every $i > 2·\dim_{ℂ} X_{\sing}$ and noticing that the
  inductive argument employed in the proof also works for $ℤ_2$-coefficients,
  the claim of Item~\ref{il:4-4-4} thus follows from smoothness in codimension
  two, Item~\ref{il:4-4-3}.

  \subsubsection*{Item~\ref{il:4-4-5}} Remaining in the analytic category,
  writing down the exponential sequences for $X$ and $X_{\reg}$,
  \[
    \begin{tikzcd}[column sep=0.35cm]
      H¹\bigl(X,\, ℤ \bigr) \ar[d, "r_1"] \ar[r] & H¹\bigl(X,\, 𝒪_X \bigr) \ar[d] \ar[r] & \Pic(X) \ar[d, hook] \ar[r, "c_1"] & H²\bigl(X,\, ℤ \bigr) \ar[d, "r_2"] \ar[r] & H²\bigl(X,\, 𝒪_X \bigr) \ar[d] \\
      H¹\bigl(X_{\reg},\, ℤ \bigr) \ar[r] & H¹\bigl(X_{\reg},\, 𝒪_{X_{\reg}} \bigr) \ar[r] & \Pic\bigl(X_{\reg}\bigr) \ar[r, "c_1"'] & H²\bigl(X_{\reg},\, ℤ \bigr) \ar[r] & H²\bigl(X_{\reg},\, 𝒪_{X_{\reg}}\bigr),
    \end{tikzcd}
  \]
  and filling in what we already know, we find a commutative diagram with exact
  rows, as follows,
  \[
    \begin{tikzcd}[column sep=1cm, row sep=1cm]
      0 \ar[r] & 0 \ar[d] \ar[r] & \Pic(X) \ar[d, hook] \ar[r, hook, two heads, "c_1\text{, iso.}"] & H²\bigl(X,\, ℤ \bigr) \ar[d, hook, two heads, "r_2\text{, iso.}"] \ar[r] & 0 \\
      0 \ar[r] & H¹\bigl( X_{\reg},\, 𝒪_{X_{\reg}} \bigr) \ar[r] & \Pic(X_{\reg}) \ar[r, two heads, "c_1"'] & H²\bigl( X_{\reg},\, ℤ \bigr) \ar[r] & 0.
    \end{tikzcd}
  \]
  The snake lemma now asserts that
  \begin{equation}\label{eq:4-4-7}
      H¹ \bigl(X_{\reg},\, 𝒪_{X_{\reg}} \bigr) ≅ \factor{\Pic(X_{\reg})}{\Pic(X)}.
  \end{equation}
  We claim that $H¹ \bigl(X_{\reg},\, 𝒪_{X_{\reg}} \bigr)$ vanishes.  For this,
  recall that the singularities of $X$ are rational, so every local ring
  $𝒪_{X,x}$ of the (holomorphic) structure sheaf has depth equal to $n$.  Since
  the singular set of $X$ has codimension at least $3$ in $X$ by
  Item~\ref{il:4-4-3}, we may apply \cite[Sec.~5, Korollar after Satz
  III]{MR0148941} or alternatively \cite[Chap.~II, Cor.~3.9 and Thm.~3.6]{BS76}
  to see that the restriction homomorphism
  \[
    H¹ \bigl(X,\, 𝒪_X \bigr) → H¹ \bigl(X_{\reg},\, 𝒪_{X_{\reg}} \bigr)
  \]
  is bijective.  However, the cohomology group on the left side was shown to
  vanish in Item~\ref{il:4-4-1} above.

  In summary, we find that every invertible sheaf on $X_{\reg}$ extends to an
  invertible sheaf on $X$.  If $D ∈ \Div(X)$ is any Weil divisor, the invertible
  sheaf $𝒪_{X_{\reg}}(D)$ will therefore extend to an invertible sheaf on $X$,
  which necessarily equals the (reflexive) Weil divisorial sheaf $𝒪_X(D)$.  It
  follows that $D$ is Cartier.  This applies in particular to the canonical
  divisor, so $X$ is $ℚ$-Gorenstein of index one.  Since $X$ is Cohen-Macaulay,
  we conclude that $X$ is Gorenstein.

  \subsubsection*{Item~\ref{il:4-4-6}} Given that $\Pic(X) = ℤ$, every line
  bundle is ample, anti-ample, or trivial; we need to exclude the case that
  $K_X$ is trivial.  But if $K_X$ were trivial, use that $X$ is Gorenstein and
  apply Serre duality to find
  \[
    h^n\bigl(X,\, 𝒪_X\bigr) = h⁰\bigl(X,\, ω_X \bigr) = h⁰\bigl(X,\, 𝒪_X \bigr) = 1.
  \]
  This contradicts Item~\ref{il:4-4-1} above.
\end{proof}

\begin{notation}[Line bundles on varieties homeomorphic to $ℙ^n$]%
  If $X$ is a projective klt variety that is homeomorphic to $ℙ^n$,
  Item~\ref{il:4-4-2} shows the existence of a unique ample line bundle that
  generates $\Pic(X) ≅ ℤ$.  We refer to this line bundle as $𝒪_X(1)$.
  Items~\ref{il:4-4-5} equips us with a unique number $r ∈ ℕ$ and such that $ω_X
  ≅ 𝒪_X(r)$.  Item~\ref{il:4-4-6} guarantees that $r ≠ 0$.
\end{notation}

\begin{rem}[Pull-back of line bundles]\label{rem:4-6}%
  The cohomology rings of $X$ and $ℙ^n$ are isomorphic.  If $φ : X → ℙ^n$ is any
  homeomorphism, then $φ^* c_1 \bigl( 𝒪_{ℙ^n}(1) \bigr) = c_1 \bigl( 𝒪_X(± 1)
  \bigr)$.  The cup products $c_1 \bigl( 𝒪_X(1) \bigr)^q$ generate the groups
  $H^{2q}\bigl( X,\, ℤ\bigr) ≅ ℤ$.
\end{rem}

\subsection{Characteristic classes}
\approvals{Daniel & yes \\ Stefan & yes \\ Thomas & yes}
\label{sec:4-2}

We have seen in Proposition~\ref{prop:4-4} that $X$ is smooth away from a closed
set of codimension $≥ 3$.  This allows defining a number of characteristic
classes.

\begin{notation}[Chern classes on varieties homeomorphic to $ℙ^n$]\label{not:4-7}%
  If $X$ is a projective klt variety that is homeomorphic to $ℙ^n$,
  Item~\ref{il:4-4-4} allows defining first and second Chern classes, as well as
  a first Pontrjagin class and a second Stiefel-Whitney class
  \begin{align*}
    c_1(X) & = r_2^{-1} c_1(X_{\reg}) ∈ H²\bigl( X,\, ℤ\bigr) \\
    c_2(X) & = r_4^{-1} c_2(X_{\reg}) ∈ H⁴\bigl( X,\, ℤ\bigr) \\
    p_1(X) & = r_4^{-1} p_1(X_{\reg}) ∈ H⁴\bigl( X,\, ℤ\bigr)\\
    w_2(X) & = r_2^{-1} w_2(X_{\reg}) ∈ H²\bigl( X,\, ℤ_2\bigr).
  \end{align*}
\end{notation}

\begin{rem}[Pontrjagin and Chern classes]\label{rem:4-8}%
  If $X$ be a projective klt variety that is homeomorphic to $ℙ^n$, the
  restriction maps $r_• : H^•\bigl( X,\, ℤ \bigr) → H^•\bigl(X_{\reg},\, ℤ
  \bigr)$ commute with the cup products on $X$ and $X_{\reg}$, which implies in
  particular that
  \[
    p_1(X) = r_4^{-1} p_1(X_{\reg}) %
    = r_4^{-1} \Bigl( c_1(X_{\reg})²-2·c_2(X_{\reg}) \Bigr) %
    = c_1(X)²-2·c_2(X) %
    ∈ H⁴\bigl( X,\, ℤ\bigr).
  \]
\end{rem}

\begin{rem}[Stiefel-Whitney class and first Chern class]\label{rem:4-9}%
  By definition and the well-known relation in the smooth case, we have
  \[
    w_2(X) = c_1(X) \ {\rm mod} 2.
  \]
\end{rem}

Novikov's result on the topological invariance of Pontrjagin classes extends to
the generalized Pontrjagin class defined in Notation~\ref{not:4-7}.

\begin{prop}[Topological invariance of Pontrjagin classes]\label{prop:4-10}%
  Let $X$ be a projective klt variety.  If $φ : X → ℙ^n$ is any homeomorphism,
  then $φ^* p_1(ℙ^n) = p_1(X)$ in $H⁴\bigl( X,\, ℤ\bigr)$.
\end{prop}
\begin{proof}
  Consider the open set $ℙ^n_{\reg} := φ(X_{\reg})$ and the restricted
  homeomorphism $φ_{\reg} : X_{\reg} → ℙ^n_{\reg}$.  Recalling from
  Item~\ref{il:4-4-4} of Propositions~\ref{prop:4-4} that the restriction maps
  \[
    r_4 : H⁴\bigl( X,\, ℤ \bigr) → H⁴\bigl(X_{\reg},\, ℤ \bigr)%
    \quad\text{and}\quad
    r_4 : H⁴\bigl( ℙ^n,\, ℤ \bigr) → H⁴\bigl(ℙ^n_{\reg},\, ℤ \bigr)%
  \]
  are isomorphic, it suffices to show that the restricted classes in rational
  cohomology agree.  More precisely,
  \begin{align*}
    && φ^* p_1(ℙ^n) & = p_1(X) && \text{in } H⁴\bigl( X,\, ℤ \bigr) \\
    ⇔ && r_4 φ^* p_1(ℙ^n) & = r_4 p_1(X) && \text{in } H⁴\bigl( X_{\reg},\, ℤ \bigr) \text{, since $r_4$'s are iso.} \\
    ⇔ && φ^*_{\reg} p_1(ℙ^n_{\reg}) & = p_1(X_{\reg}) && \text{in } H⁴\bigl( X_{\reg},\, ℤ \bigr) \text{, definition, functoriality} \\
    ⇔ && φ^*_{\reg} p_1(ℙ^n_{\reg}) & = p_1(X_{\reg}) && \text{in } H⁴\bigl( X_{\reg},\, ℚ \bigr) \text{, since } H⁴\bigl( X_{\reg},\, ℤ \bigr) = ℤ
  \end{align*}
  The last equation is Novikov's famous topological invariance of Pontrjagin
  classes, \cite{Nov65}\footnote{See \cite[Thm.~0]{MR1610975} for the precise
  result used here and see \cite[Appendix]{MR2721630} for a history of the
  result.  Igor Belegradek explains on
  \href{https://mathoverflow.net/q/442025}{MathOverflow} why compactness
  assumptions are not required.}.
\end{proof}

\begin{cor}[Relation between Chern classes on varieties homeomorphic to $ℙ^n$]\label{cor:4-11}%
  If $X$ is a projective klt variety that is homeomorphic to $ℙ^n$, then
  \[
    2·c_2(X) = \bigl[r²-(n+1)\bigr]·c_1\bigl(𝒪_X(1)\bigr)² \quad \text{in } H⁴\bigl( X,\, ℤ \bigr).
  \]
\end{cor}
\begin{proof}
  Choose a homeomorphism $φ : X → ℙ^n$, in order to compare the Pontrjagin class
  of $ℙ^n$ with that of $X$.
  \begin{align*}
    && p_1(ℙ^n) & = (n+1)·c_1\bigl(𝒪_{ℙ^n}(1)\bigr)² && \text{in } H⁴\bigl( ℙ^n,\, ℤ \bigr) \\
    ⇔ && φ^* p_1(ℙ^n) & = (n+1)·φ^*c_1\bigl(𝒪_{ℙ^n}(1)\bigr)² && \text{in } H⁴\bigl( X,\, ℤ \bigr) \\
    ⇔ && p_1(X) & = (n+1)·c_1\bigl(𝒪_X(± 1)\bigr)² && \text{Prop.~\ref{prop:4-10} and Rem.~\ref{rem:4-6}} \\
    ⇔ && c_1\bigl( 𝒪_X(r)\bigr)² - 2·c_2(X) & = (n+1)·c_1\bigl(𝒪_X(1)\bigr)² && \text{Rem.~\ref{rem:4-8}}
  \end{align*}
  The claim thus follows.
\end{proof}

Corollary~\ref{cor:4-11} allows reformulating the $ℚ$-Miyaoka-Yau inequality and
$ℚ$-Bogomolov-Gieseker inequality as inequalities between the index $r$ and the
dimension $n$.  The first remark will be relevant for varieties of general type,
whereas the second one will be used for Fano varieties.

\begin{rem}[Reformulation of the $ℚ$-Miyaoka-Yau inequality]\label{rem:4-12}%
  Let $X$ be a projective klt variety that is homeomorphic to $ℙ^n$.  Since $X$
  is smooth in codimension two, the Miyaoka-Yau inequality for $ℚ$-Chern
  classes,
  \[
    \big( 2(n+1)·\what{c}_2(X) - n·\what{c}_1(X)² \big)·[H]^{n-2} ≥ 0, \quad\text{for one ample }H,
  \]
  is equivalent to the assertion that there exists a non-negative constant $c
  ∈ ℝ^{≥ 0}$ such that
  \begin{align*}
    && \big( 2(n+1)·c_2(X) - n·c_1(X)² \big) & ≥ c·c_1\bigl( 𝒪_X(1) \bigr)² && \text{in } H⁴\bigl( X,\, ℤ\bigr) \\
    ⇔ && \big( (n+1)(r²-(n+1)) - n·r² \big)·c_1\bigl( 𝒪_X(1) \bigr)² & ≥ c·c_1\bigl( 𝒪_X(1) \bigr)² && \text{Cor.~\ref{cor:4-11}} \\
    ⇔ && \big( r²-(n+1)² \big)·c_1\bigl( 𝒪_X(1) \bigr)² & ≥ c·c_1\bigl( 𝒪_X(1) \bigr)² \\
    ⇔ && |r| & ≥ n+1.
  \end{align*}
  The Miyaoka-Yau inequality is an equality if and only if $|r| = n+1$.
\end{rem}

\begin{rem}[Reformulation of the $ℚ$-Bogomolov-Gieseker inequality]\label{rem:4-13}%
  Let $X$ be a projective klt variety that is homeomorphic to $ℙ^n$.  Since $X$
  is smooth in codimension two, the Bogomolov-Gieseker inequality for $ℚ$-Chern
  classes,
  \[
    \big( 2n·\what{c}_2(X) - (n-1)·\what{c}_1(X)² \big)·[H]^{n-2} ≥ 0, \quad\text{for one (equiv.~every) ample }H,
  \]
  is equivalent to the assertion that $|r| > n$.
\end{rem}

We will also need the topological invariance of the second Stiefel-Whitney class
$w_2$.

\begin{prop}[Topological invariance of the second Stiefel-Whitney class]%
  Let $X$ be a projective klt variety.  If $φ : X → ℙ^n$ is any homeomorphism,
  then $φ^* w_2(ℙ^n) = w_2(X)$ in $H²\bigl( X,\, ℤ/2ℤ\bigr)$.
\end{prop}
\begin{proof}
  We can argue as in the proof of Proposition~\ref{prop:4-10}, replacing
  Novikov's Theorem by the corresponding invariance result for Stiefel-Whitney
  classes due to Thom, \cite[Thm.~III.8]{Th52}.
\end{proof}

\begin{cor}[Parity of the first Chern class of varieties homeomorphic to $ℙ^n$]\label{cor:4-15}%
  If $X$ is a projective klt variety that is homeomorphic to $ℙ^n$, then $r -
  (n+1)$ is even.
\end{cor}
\begin{proof}
  This follows from the topological invariance established just above together
  with Remark~\ref{rem:4-9} and the relation $\varphi^*(c_1(𝒪_{ℙ^n}(1))) =
  c_1(𝒪_X(± 1))$.
\end{proof}

\subsection{Partial answers to Question~\ref*{q:4-3}}
\approvals{Daniel & yes \\ Stefan & yes \\ Thomas & yes}

We conclude the present Section~\ref{sec:4} with three partial answers to
Question~\ref{q:4-3}: for threefolds, we answer Question~\ref{q:4-3} in the
affirmative.  In dimension four and five, we give an affirmative answer for Fano
manifolds.  In higher dimensions, we can at least describe and restrict the
geometry of potential exotic klt varieties homeomorphic to $ℙ^n$.

\begin{prop}[Topological $ℙ^n$ with ample canonical bundle]
  Let $X$ be a projective klt variety that is homeomorphic to $ℙ^n$.  If $K_X$
  is ample, then $r > n+1$.
\end{prop}
\begin{proof}
  Recall from \cite[Thm.~1.1]{GKPT19b} that $X$ satisfies the $ℚ$-Miyaoka-Yau
  inequality.  We have seen in Remark~\vref{rem:4-12} that this implies $r = |r|
  ≥ n+1$, with $r = n+1$ if and only if equality holds in $ℚ$-Miyaoka-Yau
  inequality.  In the latter case, recall from \cite[Thm.~1.2]{GKPT19b} that $X$
  has no worse than quotient singularities.  Since quotient singularities are
  not topologically smooth, it turns out that $X$ cannot have any singularities
  at all.  By Yau's theorem (or again by \cite[Thm.~1.2]{GKPT19b}), $X$ must
  then be a smooth ball quotient, contradicting $π_1(X) = π_1(ℙ_n) = \{1\}$.
\end{proof}

\begin{prop}[Topological $ℙ^n$ with ample anti-canonical bundle]\label{prop:4-17}%
  Let $X$ be a projective klt variety that is homeomorphic to $ℙ^n$.  If $-K_X$
  is ample, then either $X ≅ ℙ^n$ or $𝒯_X$ is unstable.
\end{prop}

\begin{rem}
  Recall from \cite[Cor.~32]{KST07} that Fano varieties with unstable tangent
  bundles admit natural sequences of rationally connected foliations.  These
  might be used to study their geometry further.  If in the situation of
  Proposition~\ref{prop:4-17} we additionally assume that the index is one,
  i.e., that $r = -1$, then $Ω^{[1]}_X$ is always semistable: if $𝒮 ⊊
  Ω^{[1]}_X$ was destabilizing, then $\det 𝒮 ⊆ Ω_X^{[\rank 𝒮]}$ is either
  trivial (hence violating the non-existence of reflexive forms,
  \cite[Thm.~1]{Zh06} and \cite[Thm.~5.1]{GKKP11}) or ample (hence violating the
  Bogomolov-Sommese vanishing theorem for klt varieties,
  \cite[Thm.~7.2]{GKKP11}).
\end{rem}

\begin{proof}[Proof of Proposition~\ref{prop:4-17}]
  If $𝒯_X$ is semistable, then the $ℚ$-Bogomolov-Gieseker inequality holds, and
  we have seen in Remark~\ref{rem:4-13} that $-r = |r| > n$.  Fujita's singular
  version of the Kobayashi-Ochiai theorem, \cite[Thm.~1]{MR946238}, will then
  apply to show that $X ≅ ℙ^n$.
\end{proof}

While the Bogomolov-Gieseker inequality does not necessarily hold on a Fano
variety with unstable tangent sheaf, we still get some restriction on the index
from the following result.

\begin{prop}\label{prop:4-19}%
  Let $X$ be a projective klt variety that is homeomorphic to $ℙ^n$.  If $-K_X$
  is ample, then $r² ≥ n+1$.  In particular, if $n ≥ 4$, then $r ≥ 3$.
\end{prop}
\begin{proof}
  Since $X$ is factorial by \ref{il:4-4-5} and non-singular in codimension two
  by \ref{il:4-4-3}, we may apply \cite[Cor.~1.5]{Ou} to obtain the bound
  $c_2(X)·c_1(𝒪_X(1))^{n-2} ≥ 0$.  Then, we conclude by
  Corollary~\ref{cor:4-11}.
\end{proof}

In dimension three we can now fully answer Question~\ref{q:4-3}.

\begin{thm}[Topological $ℙ³$]\label{thm:4-20}%
  Let $X$ be a projective klt variety that is homeomorphic to $ℙ³$.  Then, $X ≅
  ℙ³$.
\end{thm}
\begin{proof}
  Since $X$ is a threefold with isolated, rational Gorenstein singularities,
  Riemann-Roch takes a particularly simple form:
  \begin{align*}
    1 \overset{\ref{il:4-4-1}}{=} χ(𝒪_X) = \frac{1}{24}·[-K_X]·c_2(X).
  \end{align*}
  With Corollary~\ref{cor:4-11}, this reads
  \[
    -48 = r·(r²-4).
  \]
  This equation has only one real solution: $r = -4$; in particular, $-K_X$ is
  ample.  As before, Fujita's theorem \cite[Thm.~1]{MR946238} applies to show
  that $X ≅ ℙ^n$.
\end{proof}

Finally, in dimensions four and five we show the following.

\begin{thm}[$\mathbb{Q}$-Fano $4$- and $5$-folds homeomorphic to projective spaces]\label{thm:4-21}%
  Let $X$ be a projective klt variety homeomorphic to $ℙ^n$, with $n=4$ or $5$.
  Assume that $K_X $ is not ample.  Then, $X ≅ ℙ^n$.
\end{thm}
\begin{proof}
  Recall that $X$ is a Gorenstein Fano variety of index $i = -r$, with canonical
  singularities, smooth in codimension two.  By \cite[Thm.~1 and 2]{MR946238},
  we may assume that $i ≤ \dim X - 1$.  Further, from
  Proposition~\ref{prop:4-19}, we see that $i ≥ 3$.  These cases have to be
  excluded.
 
  If $i = \dim X-1$, then by \cite{Fuj90}, $X$ is a hypersurface of weighted
  degree $6$ embedded in the smooth part of the weighted projective space
  $ℙ(3,2,1^{n})$.  Smooth such hypersurfaces have semistable tangent bundle by
  \cite[Prop.~6.15]{GKP22}; in particular, they satisfy the Bogomolov-Gieseker
  inequality.  Since $X$ is smooth in codimension two, the ``principle of
  conservation of numbers'', \cite[Thm.~10.2]{Fulton98}, implies that $X$
  satisfies the Bogomolov-Gieseker inequality as well, which in turn contradicts
  Remark~\ref{rem:4-13}.
 
  The remaining case, $n = 5$ and $i = 3$, is ruled out by
  Corollary~\ref{cor:4-15}, which implies that $i = -r$ has to be even. 
\end{proof}

% !TEX root = naht3

\newcommand{\etalchar}[1]{$^{#1}$}

\end{document}